\documentclass[10pt,a4paper]{article}
\usepackage{graphicx}

\usepackage{amsmath,amsfonts,amssymb}
\usepackage{amsthm}
\usepackage{tikz}
\usepackage{mathtools}
\usepackage{dsfont}
\usepackage{url}
\usepackage[affil-sl]{authblk}
\usepackage[colorlinks=true,linkcolor=blue,citecolor=green,unicode]{hyperref}
\usepackage{aliascnt}
\usepackage[T1]{fontenc}
\usepackage[utf8]{inputenc}
\usepackage[inline]{enumitem}
\usepackage{booktabs}
\setlist[enumerate]{leftmargin=*}
\usepackage{caption}
\usepackage[binary-units=true]{siunitx}

\usepackage{blkarray}

\usepackage{xspace}

\newcommand{\R}{\mathds{R}}
\newcommand{\Z}{\mathds{Z}}
\newcommand{\B}[1]{\{0,1\}^{#1}}
\newcommand{\cC}{\mathcal{C}}

\newcommand{\cO}{\mathfrak{O}}
\newcommand{\cP}{\mathcal{P}}
\newcommand{\cS}{\mathfrak{S}}
\newcommand{\cX}{\mathcal{X}}
\newcommand{\bigO}{\mathcal{O}}

\DeclareMathOperator{\conv}{conv}

\newcommand{\group}{\Gamma}
\newcommand{\perm}{\gamma}
\newcommand{\pig}{\perm\in\group}
\newcommand{\sym}[1]{\mathcal{S}_{#1}}
\newcommand{\stab}[2]{\mathrm{stab}(#1,#2)}
\newcommand{\orbit}[2]{\mathrm{orb}(#1,#2)}
\newcommand{\inv}[1]{{#1}^{-1}}

\newcommand{\st}{\,:\,}
\newcommand{\T}{^\top}
\newcommand{\sprod}[2]{{#1}\T{#2}}
\newcommand{\define}{\coloneqq}
\newcommand{\card}[1]{|#1|}
\DeclareMathOperator{\argmax}{argmax}
\DeclareMathOperator{\suc}{succ}
\newcommand{\tpgraph}{TP-graph\xspace}
\newcommand{\tpgraphs}{TP-graphs\xspace}

\theoremstyle{plain}
\newtheorem{theorem}{Theorem}[section]
\newtheorem{lemma}[theorem]{Lemma}
\newtheorem{proposition}[theorem]{Proposition}

\newtheorem*{claim*}{Claim}

\theoremstyle{definition}

\newtheorem{example}[theorem]{Example}

\allowdisplaybreaks

\begin{document}

\title{Schreier-Sims Cuts meet Stable Set: Preserving Problem Structure when Handling Symmetries}
\author{Christopher Hojny}
\affil{Eindhoven University of Technology, Netherlands,
  email: c.hojny@tue.nl}
\author{Marc E.\ Pfetsch}
\affil{TU Darmstadt, Germany,
  email: pfetsch@opt.tu-darmstadt.de}
\author{Jos\'e Verschae}
\affil{Pontificia Universidad Católica de Chile,
  email: jverschae@uc.cl}

\maketitle
\begin{abstract}
  Symmetry handling inequalities (SHIs) are a popular tool to handle
  symmetries in integer programming.
  Despite their successful application in practice, only little is known
  about the interaction of SHIs with optimization problems.
  In this article, we focus on SST cuts, an attractive class of SHIs, and
  investigate their computational and polyhedral consequences for
  optimization problems.
  After showing that they do not increase the computational  complexity of solving
  optimization problems, we focus on the stable set problem for which we
  derive presolving techniques based on SST cuts.
  Moreover, we derive strengthened versions of SST cuts and identify cases
  in which adding these inequalities to the stable set polytope maintains integrality.
  Preliminary computational experiments show that our techniques have
  a high potential to reduce both the size of stable set problems and the
  time to solve them.

\end{abstract}

\section{Introduction}

The handling of symmetries in binary programs has the goal to speed up the
solution process by avoiding the regeneration of symmetric solutions. To
fix notation, consider the binary program
$\max\,\{\sprod{c}{x} \st Ax \leq b,\; x \in \{0,1\}^n\}$, where
$A \in \Z^{m \times n}$, $b \in \Z^m$, and $c \in \Z^n$.  Let~$\sym{n}$ be the
permutation group on~$[n] \define \{1,\dots,n\}$. A permutation~$\perm$ acts
on~$x \in \R^n$ by permuting its coordinates, i.e.,
$\perm(x) \define (x_{\inv{\perm}(1)},\dots,x_{\inv{\perm}(n)})$. A
subgroup~$\group \leq \sym{n}$ is a \emph{symmetry group} of the program if
every~$\pig$ maps feasible solutions onto feasible solutions preserving
their objective values.  That is, for~$x \in \Z^n$, $A\perm(x) \leq b$ if
and only if~$Ax \leq b$, and~$\sprod{c}{x} = \sprod{c}{\perm(x)}$.

Different techniques have been suggested for symmetry handling such as
isomorphism pruning~\cite{Mar02,Mar03a,Mar03b} or adding symmetry handling
inequalities
(SHIs) \cite{Friedman2007,Hojny2020,HojnyPfetsch2019,KaibelPfetsch2008,Liberti2008,Liberti2012a};
also see Margot~\cite{Mar10} for an overview.
SHIs are systems of inequalities that turn symmetric solutions infeasible,
while keeping at least one (optimal) solution intact.

One particular way of handling symmetries is by the addition of
inequalities based on \emph{Schreier-Sims Tables} (SST). This has been proposed by
Liberti and Ostrowski~\cite{LibO14} and Salvagnin~\cite{Sal18}. The main
idea is that by iteratively computing group stabilizers, one can
handle symmetries by adding so-called \emph{SST cuts} of the form
$x_i \leq x_j$, where variable $x_i$ appears in the orbit of
variable~$x_j$, see Section~\ref{sec:SSTdefinition} for a detailed explanation.

This approach motivates our main question:
\begin{quote}
  \emph{What is the impact of adding SST
    cuts on the complexity of the underlying binary program?}
\end{quote}
Clearly,
one would hope that neither the computational nor polyhedral complexity
increases. The answer to this question is not immediate in general, since
SST cuts might change the structure of the underlying problem, in
particular, if the problem is polynomially solvable.

In this direction, we first prove in Section~\ref{sec:complexity} that
computing an optimal solution that satisfies SST cuts can be done
in polynomial time, if the underlying problem is solvable in polynomial
time. In the remaining part of the paper, we use stable set problems (and
polynomial time solvable special cases) for investigating the above
question. In Section~\ref{sec:presolving}, we elaborate on the fact that if $i$ and~$j$ are in a common clique, then the SST cut $x_i \leq x_j$
can be used to fix $x_i = 0$. Otherwise, the SST cut can sometimes
be strengthened using cliques in the orbit of $j$. Our main technical contribution is to prove that if the underlying graph is trivially perfect, i.e., a laminar
interval graph, then adding a carefully selected set of (strengthened) SST cuts and removing fixed variables
retains total unimodularity of the constraint matrix. Hence, these SST cuts do not increase the
polyhedral complexity of the problem. Interestingly, there are families of SST cuts for which total unimodularity is not preserved. In particular, this implies that different SHIs may have significant impact on the polyhedral structure of the resulting problem.
 We also study the computational impact of these inequalities in
Section~\ref{sec:experiments}.
The results indicate that the techniques of
  Section~\ref{sec:presolving} are a powerful tool to reduce graph
sizes and running times for symmetric stable set problems.

We note that related results as the ones in this paper can be obtained,
e.g., for matching or maximum flow problems.
Furthermore, Section~\ref{sec:presolving}
shows that SST cuts indeed preserve the structure of stable set
problems.  From this we derive presolving techniques that can drastically
reduce the problem size. We also find that SST cuts preserve
persistency of the edge relaxation, a helpful property exploited
in presolving.
For general independence systems, more research is need to see how our
results for stable set generalize to independence systems by
considering their conflict graph.

\section{Schreier-Sims Table Inequalities}
\label{sec:SSTdefinition}

SST cuts are SHIs derived from Schreier-Sims tables using the following algorithm.
Define the \emph{stabilizer}~$\stab{\group}{I} \define \{ \pig \st \perm(i) = i
\text{ for } i \in I\}$
of sets~$I \subseteq [n]$ and
\emph{orbits}~$\orbit{\group}{i} = \{\perm(i) \st \pig\}$ for~$i \in [n]$.
These sets can be computed in polynomial time if $\group$ is given by a set
of generators~\cite{Seress2003}. The algorithm performs the following steps, starting with~$\group' \gets \group$,
$L \gets \emptyset$:
\begin{enumerate*}[label=(\roman*)]
	\item select a \emph{leader}~$\ell \in [n] \setminus L$ and compute~$O_\ell \gets
	\orbit{\group'}{\ell}$;
	\item update~$L \gets L \cup \{\ell\}$, and~$\group' \gets
	\stab{\group'}{L}$;
	\item repeat the previous steps until~$\group'$ becomes trivial.
\end{enumerate*}

We say that each element $\ell\in L$ is a \emph{leader} and $f\in O_\ell\setminus\{\ell\}$ is a follower of $\ell$. Unless stated otherwise, we relabel the leaders such that $L=\{1,\ldots,k\}$, where $j\in L$ is the $j$th leader selected by the algorithm. One can show~\cite{LibO14,Sal18} that \emph{SST cuts}
\begin{align}
  \label{eq:SST}
  -x_{\ell} + x_{f} &\leq 0,&& \ell \in L,\; f\in O_\ell,
\end{align}
define a system of SHIs. We usually refer to a single cut by a pair $(\ell,f)$ where $\ell\in L$ and $f\in O_\ell$. Also, we define a \emph{round} as a set of SST cuts $(\ell,f)$ given by a single leader $\ell\in L$ and all its followers $f\in O_\ell$. Moreover, we denote by $S$ the set of all pairs $(\ell,f)$ for every $\ell\in L$ and $f\in O_\ell$.

A set~$S$ of SST cuts defines a \emph{symmetry
handling cone} via~\eqref{eq:SST} that we denote by~$\cC(S)$.
This cone has attained recent attention~\cite{VerschaeEtAl2021}, for
example, it has~$O(n)$ facets and defines the
closure of the set of vectors that are lexicographically maximal in
their orbits, providing the best polyhedral approximation of lexicographically
maximal vectors.
In particular, every lexicographically maximal vector in $\R^n$ satisfies the
SST cuts based on the same order of the leaders.

\section{Complexity}
\label{sec:complexity}

One drawback of symmetry handling inequalities enforcing a total
lexicographic order is that their separation problem is coNP-hard,
cf.~\cite{Hojny2018,LuksRoy2004}.
However, SST cuts are weaker, as explained at the end of the last section.
Thus, there is hope that they do not increase the computational
complexity of solving a symmetry reduced problem compared to the
original problem. This is indeed true:

\begin{theorem}
  Let~$\cX \subseteq \R^n$ and~$c \in \R^n$.
  Let~$\group \leq \sym{n}$ be a symmetry group of the
  problem~$(\cP)$ $\max\,\{\sprod{c}{x} \st x \in \cX\}$.
  Let~$S = \{(\ell,f) \st \ell\in L, f\in O_{\ell}\}$ denote a set of
  SST cuts derived from~$\group$.
  If~$(\cP)$ can be solved in~$\bigO(T(n))$ time, then an optimal solution of
  the problem $(\cP')$ $\max\,\{\sprod{c}{x} \st x \in \cX \cap \cC(S)\}$ can be found
  in~$\bigO(T(n) + \text{poly}(n))$ time.
\end{theorem}

\begin{proof}
  Let~$x$ be an optimal solution of~$(\cP)$.
  We construct an optimal solution~$x'$ of~$(\cP')$ in polynomial
  time.
  Consider the first leader $\ell=1$ and let $i_1 \in \argmax\,\{x_i \st i \in O_1\}$ and~$\pig$ be such
  that~$\perm(i_1) = 1$.
  Then, $\perm(x)$ satisfies the SST cuts $-x_1+x_f\le 0$ for~$f\in O_1$.
  By replacing $\group$ by the stabilizer of~$\ell=1$ and~$x$ by~$\perm(x)$, we can iterate the procedure for the
  remaining orbits to find a point~$x' \in \orbit{\group}{x}$ that
  satisfies all SST cuts.
  Since~$x$ is optimal, $x'$ is optimal too.
  As pointwise stabilizers can be computed in polynomial
  time~\cite{Seress2003}, $x'$  can be constructed in polynomial time.
\end{proof}
Note that we assume~$\group$ to be given by a set of generators.
Computing symmetries for integer
programs is NP-hard~\cite{Mar10}, however, so-called \emph{formulation} groups can be computed relatively fast in practice, see, e.g., \cite{PfeR19}.

\section{Presolving Reductions}
\label{sec:presolving}

In the remainder of this article, we focus on whether SST cuts preserve
problem structure.  We start by investigating how the implications of SST
cuts can be used in presolving routines.  To this end, we consider
\emph{stable set problems}: For an undirected graph~$G = (V,E)$ with node
weights~$c \in \Z^V$, find a set~$I \subseteq V$ of maximal weight such
that the elements in~$I$ are pairwise non-adjacent. The corresponding
\emph{edge formulation} is
\[
  \alpha(G) \define \max\,\Big\{ \sum_{v \in V} c_v\,x_v \st x_u + x_v \leq 1\; \forall \{u,v\}
  \in E,\; {x \in \B{V}}\Big\}.
\]
Note that all inequalities in this formulation have~$\B{}$-coefficients.
Thus, adding SST cuts changes the problem structure since SST cuts
have~$\{0, \pm 1\}$-coefficients. To overcome this issue,
we want to derive an alternative stable set problem on a graph
$G' = (V', E')$ that incorporates some implications of SST cuts.

\begin{lemma}
  Let $G = (V,E)$ be an undirected graph.
  Let~$S$ be a set of SST
  cuts for~$\alpha(G)$.
  Define~$V' = V \setminus \{v \in V \st v = f \text{ and } \{\ell, f\} \in
  E \text{ for some } (\ell, f) \in S\}$ and $G' = (V', E[V'])$, the
  induced subgraph. Then, $\alpha(G) = \alpha(G')$.
\end{lemma}
\begin{proof}
  Let~$(\ell, f)$ be a leader-follower pair.  If~$x_{f} = 1$, the SST cuts
  imply~$x_{\ell} = 1$ as well. Since at most one of them is contained in a
  stable set if~$\{\ell, f\} \in E$, $x_{f}$ can be fixed to~0, which is
  captured by $G'$.
\end{proof}

This means that we remove followers from~$G = (V,E)$
that are contained in a common edge with their leaders.
We call this operation the \emph{deletion operation}.

Note that this operation does not incorporate implications of SST
cuts~$(\ell, f)$ if~$\ell$ and~$f$ are not adjacent.  To take care of this,
we modify the graph~$G$ further. The \emph{addition operation}
adds~$\{v, f\}$ for every neighbor~$v$ of~$\ell$ to~$E$.  Doing so,
setting~$x_f = 1$ forces~$x_v = 0$ for all neighbors~$v$ of~$\ell$.

\begin{proposition}
  Let~$G = (V,E)$ be an undirected graph with weights~$c \in \Z^V$.
  Let~$G' = (V', E')$ arise from~$G$ by applying deletion and addition
  operations for a set of SST cuts.
  Suppose~$c_v \neq 0$ for all~$v \in V$.
  Then, every weight maximal stable set in~$G'$ is weight maximal in~$G$
  and satisfies all SST cuts.
\end{proposition}

\begin{proof}
  Since the deletion operation incorporates implications of SST cuts
  into~$G'$, it cannot remove all optimal solutions.  The missing
  implications of SST cuts are that setting~$x_f = 1$ for a follower~$f$
  implies~$x_\ell = 1$ for the corresponding leader~$\ell$.  If~$x_f = 1$,
  then the edges introduced by the addition operation cause~$x_v = 0$ for
  all neighbors~$v$ of~$\ell$.  Hence, if~$c_\ell > 0$, $x_\ell = 1$ in an
  optimal solution if~$x_f = 1$.  Moreover, if~$c_\ell < 0$, then $x_f$ is
  not set to 1 in an optimal solution, since~$c_f = c_\ell < 0$
  because~$\ell$ and~$f$ are symmetric.  Finally, note that
  setting~$x_v = 1$ for some neighbor~$v$ of~$\ell$ causes~$x_f = 0$
  and~$x_\ell = 0$.  Thus, exactly the implications of SST cuts are
  incorporated by the deletion and addition operation, which keeps at least
  one optimal solution intact.
\end{proof}

The previous result has an important implication for the edge formulation:
SST cuts preserve persistency.  Persistency is an important property, which
says that if an optimal solution of the LP relaxation of the edge
formulation has an integral coordinate, there exists an optimal
\emph{integral} solution of the stable set problem with the same integral
coordinate~\cite{NemhauserTrotter1975}.  This property can be used as a
presolving routine for stable set problems to remove some nodes and edges.
On top of this, the deletion and addition operations can be used as another
symmetry-based presolving routine, \emph{SST presolving}.  While the
deletion operation decreases the problem size, which is expected to have a
positive impact on solving time, the addition operation introduces new
edges.  Since these edges handle symmetries, one might expect
that the addition operation has a positive impact on solving time, which is
confirmed computationally in Section~\ref{sec:experiments}.

\section{Strengthened SST Cuts}
\label{sec:cuts}

The edge formulation is known to provide a poor LP-bound on the true
weight of a maximum stable set.
To strengthen this formulation, many cutting
planes such as odd cycle or odd wheel inequalities have been
derived~\cite{Padberg1973}.
One of the most important classes of inequalities, however, are \emph{clique
  inequalities}~$\sum_{v \in C} x_v \leq 1$ for all cliques $C \subseteq V$ in~$G$.
These inequalities define facets of the \emph{stable set polytope}~$P(G)$ if and
only if~$C$ is an inclusionwise maximal clique~\cite{GrotschelEtAl1981}.
The aim of this section is to investigate the following strengthening of SST cuts:

\begin{lemma}
  Let~$G = (V,E)$ be an undirected graph and
  let~$(\ell,f)$ for $f\in O_{\ell}$ be SST cuts derived for~$P(G)$.  Then,
  the following
  \emph{SST clique cut} is an SHI for every clique~$C \subseteq O_{\ell}$:
  \begin{equation}\label{eq:SSTclique}
    -x_\ell + \sum_{f \in C} x_f \leq 0.
  \end{equation}
\end{lemma}
\begin{proof}
  If~$x_f = 1$ for some~$f \in C$, the SST cuts imply~$x_\ell = 1$.
  Since~$C$ forms a clique, at most one follower~$f$ can have~$x_f = 1$,
  concluding the proof.
\end{proof}

Note that SST clique cuts generalize SST cuts, because a single follower
always defines a clique.

One can show for a single round of SST cuts~$S$ that an SST clique cut defines a facet of~$P(G,S) \define \conv\{ x
\in \B{V} \st x \in P(G) \cap \cC(S)\}$, the
\emph{symmetry-reduced stable set polytope}, if the clique is maximal in~$O_{\ell}$ and no~$f \in C$ is adjacent with~$\ell$.
Moreover, SST clique cuts are applicable to general independence systems by
defining them based on the conflict graph~$\bar{G}$.

 Since SST clique cuts are based on
cliques, we restrict our investigation to graphs~$G$ for which~$P(G)$ is
completely described by clique and non-negativity inequalities: perfect
graphs~\cite{Golumbic2004}.  In general, adding SST clique cuts does not
provide a complete description of~$P(G,S)$, e.g., if~$G$ is a $4$-cycle and~$S$
contains all possible SST cuts.  Therefore,
we restrict ourselves to perfect graphs~$G$ such that~$P(G)$ is described
by a totally unimodular constraint matrix: interval
graphs~\cite{Golumbic2004}.

An undirected graph~$G = (V,E)$ is called \emph{interval graph} if, for
all~$v \in V$, there is a real interval~$I_v$ such that, for all
distinct~$u$, $v \in V$, we have~$\{u,v\} \in E$ if and only if~$I_u \cap I_v
\neq \emptyset$.
A graph is called \emph{trivially perfect (TP)} if it is an interval graph whose
interval representation can be chosen to be laminar, i.e., if the intervals intersect, one is contained in the other.
Let~$\cC = \cC(G)$ be the set of maximal cliques of the undirected
graph~$G$.
Then, the \emph{clique matrix} $M(G)$ of~$G$ is the~$\cC \times V$-dimensional
clique-node incidence matrix of~$G$.
The clique matrix of interval graphs, and thus trivially perfect graphs, is
totally unimodular~\cite{Golumbic2004}. For a given graph $G$, let $\group_G$ be its automorphism group. Similarly, $\group$ denotes the symmetry group of the stable set program for graph $G$ as defined in Section~\ref{sec:presolving}. Note that $\group$ is a subgroup of $\group_G$ whose permutations also preserve node weights.

W.l.o.g.\ we assume that all intervals in the interval
representation~$(I_v)_{v \in V}$ of a \tpgraph are pairwise different.
We derive a \emph{rooted forest representation}~$T_G = (V,A)$ for a \tpgraph~$G= (V,E)$, where~$(u,v) \in A$ if and only if~$I_v
\subsetneq I_u$ and there is no~$w \in W$ with~$I_v \subsetneq I_w
\subsetneq I_u$.

One natural question is whether adding SST clique cuts to a complete
description of $P(G)$ provides a complete description of $P(G,S)$ if $G$ is
a \tpgraph.  In the remainder of this section, we give an answer by
providing a sufficient criterion on when adding SST clique cuts preserves total unimodularity of
the clique matrix of \tpgraphs. Moreover, for
\tpgraphs, the number of maximal cliques is linear in the number of nodes.  Picking
up our motivational question from the introduction, this shows that there
is a polynomial sized complete linear description of $P(G,S)$ and thus the
polyhedral complexity is not increased.

To derive our sufficient criterion, we introduce the notion of stringent SST cuts.
Let~$L = \{1,\dots,k\}$ be the leaders of a family of SST cuts~$S$.
Note that the orbits~$O_1,\dots,O_k$ define a laminar family.
For~$\ell \in L$, let~$\mathcal{M}(\ell)\subseteq \{O_1,\dots,O_{\ell}\}$ be the collection of inclusionwise maximal sets in $\{O_1,\dots,O_{\ell}\}$. 
The family~$S$ of SST cuts is called \emph{stringent} if the orbit~$O_\ell$
of leader~$\ell \in L$ is computed using the group~$\stab{\group}{[\ell-1]}
\cap \stab{\group}{\cO_\ell}$, where~$\cO_\ell = \bigcup_{O\in \mathcal{M}(\ell-1):\ell \notin O} O$, that is, the group must also stabilize all maximal orbits not containing $\ell$.
  
That is, stringent SST cuts not only require to stabilize previous
leaders, but also entire orbits if they do not contain the current
leader.
\begin{example}
  \label{ex:stringency}
  Figure~\ref{fig:stringent} shows the tree representation of a \tpgraph.
  The set of SST cuts for orbits~$O_1 = \{1,\dots,6\}$ and~$O_7 =
  \{7,\dots,14\}$ with leaders~$L=\{1,7\}$ (without relabeling) is not
  stringent, because~$7 \notin O_1$.
  Hence, $O_1$ needs to be stabilized, which reduces~$O_7$ to~$\{7,8\}$ for
  stringent SST cuts.
  Another example of stringent SST cuts is given by
  the leaders $1$, $3$, $5$ (in that order) and orbits~$O_1 = \{1,\dots,6\}$, $O_3 = \{3,\dots,6\}$, $O_5 = \{5,6\}$, because~$3$, $5 \in O_1$.
\end{example}
\begin{figure}[tb]
  \begin{minipage}[t]{0.5\linewidth}
    \centering
    \begin{tikzpicture}
      \tikzstyle{s} += [circle,draw=black,inner sep=0pt,minimum size=2mm];
      \node (1) at (0,0) [s] {};
      \node (2) at (.5,0) [s] {};
      \node (3) at (1,0) [s] {};
      \node (4) at (1.5,0) [s] {};
      \node (5) at (2,0) [s,label=above:7] {};
      \node (6) at (2.5,0) [s,label=above:8] {};
      \node (7) at (3,0) [s,label=above:9] {};
      \node (8) at (3.5,0) [s,label=above:10] {};
      \node (9) at (4,0) [s,label=above:11] {};
      \node (10) at (4.5,0) [s,label=above:12] {};
      \node (11) at (5,0) [s,label=above:13] {};
      \node (12) at (5.5,0) [s,label=above:14] {};

      \node (13) at (.25,-1) [s,label=left:1] {};
      \node (14) at (1.25,-1) [s,label=left:2] {};
      \node (15) at (2.25,-1) [s,label=left:3] {};
      \node (16) at (3.25,-1) [s,label=left:4] {};
      \node (17) at (4.25,-1) [s,label=left:5] {};
      \node (18) at (5.25,-1) [s,label=left:6] {};

      \node (19) at (0.75,-2) [s] {};
      \node (20) at (2.75,-2) [s] {};
      \node (21) at (4.75,-2) [s] {};

      \node (22) at (2.75,-3) [s] {};

      \draw[->] (13) -- (2);
      \draw[->] (14) -- (4);
      \draw[->] (15) -- (6);
      \draw[->] (16) -- (8);
      \draw[->] (17) -- (10);
      \draw[->] (18) -- (12);
      \draw[<-] (1) -- (13);
      \draw[<-] (3) -- (14);
      \draw[<-] (5) -- (15);
      \draw[<-] (7) -- (16);
      \draw[<-] (9) -- (17);
      \draw[<-] (11) -- (18);

      \draw[->] (19) -- (14);
      \draw[->] (20) -- (16);
      \draw[->] (21) -- (18);
      \draw[<-] (13) -- (19);
      \draw[<-] (15) -- (20);
      \draw[<-] (17) -- (21);

      \draw[<-] (19) -- (22);
      \draw[<-] (20) -- (22);
      \draw[<-] (21) -- (22);
    \end{tikzpicture}
      \captionof{figure}{Example for (non-) stringent SST cuts.}
      \label{fig:stringent}
  \end{minipage}
  \hfill
  \begin{minipage}[t]{0.4\linewidth}
    \centering
    \begin{tikzpicture}
      \tikzstyle{s} += [circle,draw=black,inner sep=0pt,minimum size=2mm];
      \node (1) at (0,0) [s] {};
      \node (2) at (2,0) [s] {};
      \node (3) at (1,1) [s] {};
      \node (4) at (0,1) [s,label=left:$\ell$] {};
      \node (5) at (2,1) [s,label=right:$f$] {};
      \draw[-] (1) -- (2) -- (3) -- (1);
      \draw[-] (1) -- (4);
      \draw[-] (2) -- (5);
      \node (fake) at (1,-1) {};
    \end{tikzpicture}
    \captionof{figure}{An interval graph and SST cut $(\ell, f)$.}
    \label{fig:interval}
  \end{minipage}
\end{figure}

Although stringency seems to be restrictive, we can implement the algorithm in Section~\ref{sec:SSTdefinition} so that it always generates
stringent SST cuts. Indeed, we can maintain the following property: if in Step~(i) a given leader~$\ell$ is selected, then the following leaders need to be selected from~$O_{\ell}$ first.
Once all elements of~$O_{\ell}$ have been considered as leaders, the
group $\stab{\group}{[\ell]}$ stabilizes~$O_{\ell}$, and we can continue with a next leader $\ell' \in [n] \setminus O_{\ell}$. Hence, we obtain stringent SST cuts by choosing leaders in a depth-first search fashion.

We are now able to formulate the main result of this section.

\begin{theorem}
  \label{thm:SSTclique}
  Let~$G = (V,E)$ be a \tpgraph.
  Consider a set~$S$ of stringent SST cuts.
  The matrix that arises by applying the following two operations is
  totally unimodular:
  \begin{enumerate}[topsep=0ex]
  \item adding all SST clique cuts derivable from~$S$;
  \item deleting columns whose nodes get deleted by the deletion
    operation.
  \end{enumerate}
\end{theorem}
In general, this theorem does not hold if we drop stringency, because
experiments with the code from~\cite{WalterT13} show that the
non-stringent SST (clique) cuts from Example~\ref{fig:stringent} do not
preserve total unimodularity when adding them to the clique matrix of the
corresponding \tpgraph.
Moreover, since SST clique cuts dominate SST cuts, it is necessary to
replace SST cuts by clique cuts.
Also, the requirement of \tpgraphs and to apply the deletion operation
are necessary for the validity of the theorem: Figure~\ref{fig:interval} shows
an interval graph that is not TP and an SST cut such that the extended
clique matrix is not totally unimodular; if there is an edge~$\{\ell, f\}$
in~$G$ for an SST cut $(\ell, f)$, then the extended clique matrix contains
a $2 \times 2$-submatrix with rows~$[1,1]$ and~$[-1,1]$, i.e., with
determinant~2.

To prove Theorem~\ref{thm:SSTclique}, we proceed in two steps.
We reduce the case of SST clique cuts to SST cuts, and then show that the
result holds for this simple case.

\subsection{Reduction to a Simple Case}

We exploit the symmetry group structure of \tpgraphs to reduce the
discussion of SST clique cuts to SST cuts.
Consider a \tpgraph~$G = (V,E)$ with automorphism group~$\group_G$ and forest
representation~$T$.
A \emph{chain} in~$T$ is a directed path~$c$ with terminal node~$u$
such that the out-degree $\delta^+_T(v)$ with respect to $T$ equals~1, for every node $v$ of $c
\setminus\{u\}$.

\begin{lemma}
\label{lem:chainDecomp}
  Let~$G = (V,E)$ be a \tpgraph.
  For any node~$v \in V$, the induced subgraph of $T_G$ in $\orbit{\group_G}{v}$ decomposes into chains of
  the same length and~$\group_G$ acts independently on each chain like the
  symmetric group.
\end{lemma}
\begin{proof}
  The nodes~$w$ in a chain~$c$ are pairwise interchangeable as exchanging
  their corresponding intervals~$I_w$ does not change the adjacency structure.
  Therefore, $\group_G$ acts on~$c$ as the symmetric group, while fixing the
  remaining nodes of~$G$.
  Moreover, if a path~$c$ is not a chain, then there exist two distinct nodes~$u$
  and~$v$ in~$c$ with out-degree at least~2.
  If~$v$ appears before~$u$ in~$c$, the degree of~$v$ in~$G$
  is larger than the degree of~$u$.
  Hence, they cannot be symmetric.
  Therefore, for any~$v \in V$, $\orbit{\group_G}{v}$ decomposes into
  chains.
  They need to have the same length because the corresponding paths need to
  be symmetric.
\end{proof}
When applying SST cuts to the stable set problem, we are using a subgroup \mbox{$\group \le \group_G$} that also preserves node weights.
In this case, we can sort the nodes along each chain consecutively by
their node weights because~$\group_G$ acts like the symmetric
group on each chain. That is, Lemma~\ref{lem:chainDecomp} also holds for the subgroup~$\group$.
We use this observation to define an auxiliary graph~$G_S$ for a
family of SST cuts~$S$.
Whenever we compute an orbit~$O_\ell$, we also compute its chain
decomposition according to the current stabilizer group.
After computing all decompositions, $G_S$ arises from~$G$ as follows.
Each chain computed in the decomposition of the orbits is contracted into a
single node.
If the chain contains a leader, then we give the contracted node the lowest
label of a leader within the chain.
Otherwise, we give the contracted node an arbitrary label within the chain.

The interpretation of this graph is as follows.
If a chain contains a leader, the deletion operation allows
us to remove all nodes except for the leader from the graph.
If a chain contains several leaders, it is only necessary to keep the leader
considered first.
For a chain~$c$ that does not contain a leader, the columns for~$v \in c$ of
the clique-node adjacency matrix are identical.
This is true, because we never compute subchains of already considered
chains, because the symmetry group acts independently like the symmetric
group on each chain, i.e., we can always exchange all nodes within a chain if
none of them is stabilized.
They are in particular still identical if we add SST clique cuts to the
matrix, because chains in~$T$ correspond to cliques in~$G$.
Hence, for deciding total unimodularity, we can remove symmetric columns.

We can now reduce Theorem~\ref{thm:SSTclique} to the case of simple SST
cuts by applying the following lemma.

\begin{lemma}
  \label{lem:reduction}
  Let~$G = (V,E)$ be a \tpgraph and let~$S$ be a set of SST clique
  cuts.
  Then, the matrix obtained by
  \begin{enumerate}[topsep=0ex]
  \item adding SST clique cuts for~$S$ to the clique matrix~$M(G)$ and
  \item deleting columns contained in SST cuts for~$S$ such that the
    corresponding leader and follower are adjacent,
  \end{enumerate}
  is totally unimodular if and only if the matrix $A_S$ obtained by extending $M(G_S)$
  with the simple SST cuts corresponding to~$S$ in~$G_S$ is totally
  unimodular.
\end{lemma}
\begin{proof}
  By the previous discussion, the matrix~$A_S$ is a
  submatrix of the extended clique matrix~$A$ of~$G$.
  Thus, if~$A$ is totally unimodular, so is~$A_S$.

  For the other direction, assume~$A_S$ is totally unimodular.
  To see that also~$A$ is totally unimodular, select an arbitrary square
  submatrix~$B$ of~$A$.
  If~$B$ does not contain a row corresponding to an SST clique cut, $B$ is
  a submatrix of~$M(G)$, and thus totally unimodular.
  For this reason, assume~$B$ contains a row corresponding to an SST clique
  cut.
  Select an SST clique cut in~$B$ whose leader~$\ell$ has the largest
  value.
  Let~$C$ be the corresponding clique.
  If~$B$ contains two columns corresponding to nodes~$v$ and~$w$ in~$C$,
  then these columns are identical by the previous discussion.
  Consequently, $\det(B) = 0$.

  Thus, suppose~$B$ contains only one column corresponding to a
  node~$v$ in~$C$.
  If the column corresponding to~$\ell$ is not present in~$B$, we
  expand~$\det(B)$ along the row corresponding to the SST clique cut.
  Since this row contains exactly one~1-entry, we find~$\det(B) \in \{0, \pm
  1\}$ by applying the above arguments inductively.
  Therefore, we may assume that, for each selected SST clique cut
  in~$B$, there is at most one column~$v$ that contains a node from the
  corresponding clique of the SST clique cut.
  Hence, $B$ is a submatrix of~$A_S$ and~$\det(B) \in \{0, \pm
  1\}$ follows.
\end{proof}
\subsection{Proving the Simple Case}

Consequently, Theorem~\ref{thm:SSTclique} holds if the following theorem
holds.

\begin{theorem}
  \label{thm:stringentSST}
  Let~$G = (V,E)$ be a \tpgraph.
  Consider a set of stringent SST cuts for leaders~$L=[k]$ and
  orbits~$O_1,\dots,O_k$.
  If no orbit contains an edge from~$E$, then the clique matrix~$M(G)$
  extended by the simple SST cuts is totally unimodular.
\end{theorem}

In the remainder of this section, we prove Theorem~\ref{thm:stringentSST}.
Let~$T = T_G$ be the forest representation of a \tpgraph~$G$. We denote the set of all paths in~$T$ that connect a root node~$r$ with a leaf
by~$\cP$.
The paths in $\cP$ that contain~$v \in V$ are denoted by~$\cP_v$ and are in a
one-to-one correspondence with the cliques~$\cC_v \subseteq \cC$ in~$G$
that contain~$v$.
We call a set of nodes~$S \subseteq V$ \emph{path-disjoint} if~$\cP_u \cap
\cP_v = \emptyset$ for all distinct~$u$, $v \in S$.
Note that there is a one-to-one correspondence between path-disjoint sets
in~$T$ and stable sets in \tpgraphs~$G$.

We define a \emph{reduction operation} as follows:
For a set~$S \subseteq V$ and~$v \in S$, let~$d$ be a node on the
unique~$r$-$v$-path, where $r$ is the unique root of the connected component containing $v$ in~$T$.
If we delete~$d$ from~$T$, then~$T$ decomposes into connected components
that are rooted trees.
The \emph{reduced graph}~$T_d(S)$ is the graph defined by the connected
components whose roots are children of~$d$ and that do not contain any node
from~$S$.
We also need the following property.
A family of path-disjoint sets~$S_1,\dots,S_k \subseteq V$ has
the \emph{recursion property} if
\begin{enumerate}[label=(RP\arabic*),topsep=1ex]
\item $S_1,\dots, S_k$ are pairwise disjoint, and

\item\label{RP2} for every~$i \in [k-1]$, there exists~$d^i \in V$ such
  that~$S_{i+1} \subseteq T_{d^{i}}(\bigcup_{j = 1}^{i-1}S_j)$.
\end{enumerate}
If~$\cS = \{S_1, \dots, S_k\}$ is a laminar family of subsets of~$V$, we
define, for~$S \in \cS$, $\cS_S \define \{ S' \in \cS \st S' \subsetneq S\}$.
We say that the laminar family~$\cS$ has the \emph{laminar recursion
  property} if
\begin{enumerate}[label=(LRP\arabic*),topsep=1ex]
\item\label{LRP1} for all~$S \in \cS$, there is~$u_S \in S$ not
  contained in any set of~$\cS_S$, and
\item\label{LRP2} the inclusionwise maximal sets in~$\cS$ have the
  recursion property.
\end{enumerate}
Similarly, $\cS$ has the (laminar) recursion property with respect to a \tpgraph~$G$, if it has the same property for its tree representation $T_G$.
Using these concepts, we can prove Theorem~\ref{thm:stringentSST}.
In fact, we show a stronger result for general ordering inequalities~$x_u
\geq x_v$ that are not necessarily based on symmetries.

\begin{theorem}
  \label{thm:LRP}
  Let~$G = (V,E)$ be a \tpgraph and let~$S_1, \dots, S_k
  \subseteq V$ be stable sets satisfying the laminar recursion property.
  For each~$i \in [k]$, let~$u_i \in S_i$ adhere to~\ref{LRP1}.
  Then, the clique matrix~$M(G)$ extended by the ordering inequalities
  $x_{u_i} \geq x_v$ for all~$i \in [k]$ and~$v \in S_i \setminus \{u_i\}$
  is totally unimodular.
\end{theorem}

To prove this theorem, we need the following lemmata and concepts.

Let~$\cP$ be the set of root-leaf paths of a rooted forest~$T = (V,A)$.
We identify each path in~$\cP$ by its unique leaf node.
For a node~$v \in V$, we denote by~$\suc(v)$ the set of direct successors of~$v$
in~$T$, i.e., $\suc(v) = \{w \in V \st (v,w) \in A\}$.
If~$P \subseteq \cP$ is a set of paths, we denote by~$P_v$ the set of all
paths containing~$v \in V$.
Note that, if~$v$ is a leaf, then~$P_v = \{v\}$. Otherwise, $P_v =
\bigcup_{w \in \suc(v)} P_w$.

An \emph{equicoloring} (equitable bicoloring) of~$P \subseteq \cP$ is a partition~$P^+ \cup P^-$
of~$P$ such that, for every~$v \in V$,
$
  \delta_v \define \card{P^+_v} - \card{P^-_v} \in \{0, \pm 1\}.
$
Due to the forest structure of~$T$, for each~$v \in V$ that is not a
leaf, we have $\delta_v = \sum_{w \in \suc(v)} \delta_w$.
\begin{lemma}
  \label{lem:recursiveEBC1}
  Let~$T = (V,A)$ be a rooted tree with root~$r$
  and let~$P \subseteq \cP$.
  Let~$S \subseteq V$ be non-empty, path-disjoint and
  suppose that~$P_v \neq \emptyset$ for each~$v \in S$.
  Then, there exists an equicoloring~$P^+ \cup P^-$ of~$P$ such
  that
  \[
    \sum_{v \in S} \delta_v \in
    \begin{cases}
      \{-1\}, & \text{if } \delta_r = -1,\\
      \{0,1\}, & \text{if } \delta_r \in \{0,1\}.
    \end{cases}
  \]
\end{lemma}
\begin{proof}
  We proceed by induction on the height~$h$ of~$T$.
  If~$h = 1$, then~$T$ consists just of the root node~$r$ and~$S \subseteq \{r\}$.
  Moreover, if~$r \in S$, then~$S = \{r\}$ since every path in~$\cP$
  contains~$r$.
  In both cases, we can choose $P^- = \{r\}$ and $P^+ = \emptyset$ and the assertion holds.
  
  If~$h > 1$ and~$r \notin S$, consider the forest~$T'$ that arises by
  removing~$r$ and all  its outgoing arcs from tree~$T$.
  The height of~$T'$ is~$h - 1$.
  Thus, if~$T'$ is connected, the assertion follows by induction.
  Otherwise, $T'$ has~$k > 1$ connected components which are rooted trees.
  Let~$r_1,\dots, r_k$ be the corresponding root nodes and, for~$i \in
  [k]$, let~$S_i$ be the nodes in~$S$ that are descendants of~$r_i$.
  By the inductive hypothesis, we can find for every connected component~$i
  \in [k]$ an equicoloring such that~$\sum_{v \in S_i} \delta_v =
  \delta_{r_i}$, or~$\sum_{v \in S_i} \delta_v = 0$ and~$\delta_{r_i} = 1$,
  or~$\sum_{v \in S_i} \delta_v = 1$ and~$\delta_{r_i} = 0$.
  
  Let~$C^-$ be the connected components~$i$ with~$\delta_{r_i} =
  -1$, let~$C^+$ be the connected components~$i$ with~$\delta_{r_i} =
  \sum_{v \in S_i} \delta_v = 1$, let~$C^0$ be the connected
  components~$i$ with~$\delta_{r_i} = 1$ and~$\sum_{v \in S_i} \delta_v =
  0$, and let~$C_0$ be the connected components~$i$ with $\delta_{r_i} = 0$
  and~$\sum_{v \in S_i} \delta_v = 1$.
  After possibly changing the two classes of the equicoloring for
  some components, we can assume~$\card{C^+} - \card{C^-} \in \{0,1\}$.
  Combining these equicolorings for the components in~$C^+$
  and~$C^-$ gives us an equicoloring of~$C^+ \cup C^-$
  with
  \[
    \Delta \define
    \sum_{i \in C^+ \cup C^-} \delta_{r_i}
    =
    \sum_{i \in C^+ \cup C^-} \sum_{v \in S_i} \delta_v \in \{0,1\}.
  \]
  First, suppose~$C_0 = \emptyset$.
  If~$C^0 \neq \emptyset$, let~$C_1,\dots,C_\ell$ be an ordering
  of the components in~$C^0$ with corresponding roots~$r(1),\dots,r(\ell)$.
  We distinguish whether~$\Delta = 0$ or~$\Delta = 1$.
  If~$\Delta = 0$, we flip the two classes of the equitable partitions
  in~$C^0$ with an even label; if~$\Delta = 1$, we flip the classes for
  partitions with an odd label.
  Then,
  \[
    \sum_{i = 1}^k \delta_{r_i}
    =
    \Delta + \sum_{i = 1}^\ell \delta_{r(i)}
    =
    \begin{cases}
      0, & \text{if } \Delta = 0 \text{ and } \ell \text{ is even,} \text{ or } \Delta = 1 \text{ and } \ell \text{ is odd,}\\
      1, & \text{if } \Delta = 1 \text{ and } \ell \text{ is even,} \text{ or } \Delta = 0 \text{ and } \ell \text{ is odd.}\\
    \end{cases}
  \]
  That is, $\sum_{i = 1}^k \delta_{r_i} \in \{0,1\}$ and~$\sum_{i = 1}^k
  \sum_{v \in S_i} \delta_v \in \{0,1\}$.

  Second, if~$C_0 \neq \emptyset$, we proceed as before to find an
  equicoloring of the components in~$\bar{C} \define C^+ \cup C^-
  \cup C^0$
  with~$\sum_{i \in \bar{C}} \delta_{r_i} \in \{0,1\}$ and~$\sum_{i \in \bar{C}}
  \sum_{v \in S_i} \delta_v \in \{0,1\}$.
  Since the connected components from~$C_0$ do not affect the value
  of~$\sum_{i \in \bar{C}} \delta_{r_i}$, we can flip the classes of the
  bicoloring of every second component in~$C_0$ to maintain the property
  that~$\sum_{i = 1}^k \sum_{v \in S_i} \delta_v \in \{0,1\}$.

  To conclude the proof, we extend the equicolorings of the
  individual connected components of~$T'$ to an equicoloring
  of~$T$ by associating each path in~$T'$ by the corresponding path in~$T$.
  Then, for every~$v \in V \setminus \{r\}$, $\delta_v \in \{0, \pm 1\}$
  follows trivially.
  Moreover, since~$\delta_r = \sum_{i = 1}^k \delta_{r_i}$, also~$\delta_r
  \in \{0, \pm 1\}$.
  In particular, $\delta_r$ has the desired relation with~$\sum_{v \in S}
  \delta_v$ by the above argumentation.
\end{proof}
\begin{lemma}
  \label{lem:recursiveEBC}
  Let~$T = (V,A)$ be a rooted tree,
  let~$S_1, \dots, S_k \subseteq V$ have the recursion property, and let~$P
  \subseteq \cP$.
  If~$P_v \neq \emptyset$ for each~$v \in \bigcup_{i = 1}^k S_i$, then
  there exists an equicoloring~$P^+ \cup P^-$ of~$P$ such
  that~$\sum_{v \in S_i} \delta_v \in \{0, \pm 1\}$ for all~$i \in [k]$.
\end{lemma}
\begin{proof}
  We prove the assertion by induction on~$k$.
  If~$k = 1$, the statement follows from Lemma~\ref{lem:recursiveEBC1}.
  Inductively, we can thus assume that there is an equicoloring~$P^+ \cup
  P^-$ of~$P$ that has the desired properties
  for~$S_1, \dots, S_{k-1}$, and show that we can adapt it to such an
  equicoloring for~$S_1,\dots, S_k$.

  Let~$d^{k-1}$ adhere to~\ref{RP2}, and let~$T_k = T_{d^{k-1}}(\bigcup_{i
    = 1}^{k-2} S_i)$.
  W.l.o.g.\ we can assume that~$T_k$ consists of a single connected
  component.
  Otherwise, we show the result for the graph~$T'$ in which we replace the
  arcs from~$d^{k-1}$ to the roots of~$T_k$ by a single arc~$(d^{k-1}, d')$
  and connect~$d'$ with the roots of~$T_k$.
  The equicoloring found for~$T'$ is then also an equicoloring
  for~$T$ with the same properties.

  Let~$r^k$ be the root of~$T_k$.
  The equicoloring~$P^+ \cup P^-$ derived for~$S_1, \dots, S_{k-1}$
  yields~$\delta_{r^k} \in \{0, \pm 1\}$.
  If~$\card{P_{r^k}}$ is even, then~$\delta_{r^k}$ is necessarily~0 in
  every equicoloring.
  Analogously, if~$\card{P_{r^k}}$ is odd, then~$\delta_{r^k} = \pm 1$ in
  every equicoloring.
  By Lemma~\ref{lem:recursiveEBC1}, there exists an equicoloring
  of~$P_{r^k}$ with~$\delta^k$-values such that~$\sum_{v \in S_k}
  \delta^k_v \in \{0, \pm 1\}$.
  By the previous observation, $\delta^k_{r^k} = 0$ if and only
  if~$\delta_{r^k} = 0$.
  Thus, after possibly flipping the two classes found in the equicoloring
  of~$P_{r^k}$, $\delta^k_{r^k} = \delta_{r^k}$.
  Consequently, if we change the equicoloring~$P^+ \cup P^-$
  on~$P_{r^k}$ such that it coincides with the bicoloring found for~$T_k$,
  it is still an equicoloring for~$T$.
  It satisfies~$\sum_{v \in S_k} \delta_v \in \{0, \pm 1\}$, and the values
  of~$\sum_{v \in S_i} \delta_v$ for~$i \in [k-1]$ did not change, because~$T_k$
  does not contain any node from~$\bigcup_{i = 1}^{k-1} S_i$.
  That is, we have found the desired equicoloring.
\end{proof}
\begin{proof}[Proof of Theorem~\ref{thm:LRP}]
  We may assume that~$G$ is connected and that no node involved in an ordering
  inequality is the root node of the tree representation of~$G$:
  Otherwise, we introduce a node~$w$ that is connected with all nodes
  in~$G$, yielding a graph~$G'$.
  Moreover, we can recover the assertion for~$G$ from~$G'$, because the
  extended clique matrix of~$G$ is a submatrix of the extended clique
  matrix of~$G'$.

  In the following, we work with the tree representation of~$G$.
  Let~$P \subseteq \cP$ be a set of paths in the tree representation.
  Let~$D$ be a set of ordering inequalities encoded via the leader-follower
  pair~$(u,v)$, and let~$U \define \{u_1, \dots, u_k\}$.
  That is, $D_{u,v}$ is the inequality~$x_u \geq x_v$.
  To show that $M(G)$ extended by
  ordering inequalities is totally unimodular, we use Ghouila-Houri's
  equicoloring criterion~\cite{GH1962}.
  That is, we need to find partitions~$P^+ \cup P^-$ of~$P$ and $D^+ \cup D^-$
  of~$D$ such that
  \[
    \Delta_w
    =
    \card{P^+_w} - \card{P^-_w}
    +
    \sum_{u \in V}(\card{D^+_{u,w}} - \card{D^-_{u,w}})
    +
    \sum_{v \in V}(\card{D^-_{w,v}} - \card{D^+_{w,v}}) \in \{0,\pm 1\}
  \]
  for all~$w \in V$. 
  Our strategy is to show the statement for the case that all
  sets~$S_1,\dots,S_k$ are pairwise disjoint first.
  Afterwards, we will use this result as an anchor for the general case.
  The anchor allows us to derive a result for the inclusionwise maximal
  sets among~$S_1,\dots, S_k$.
  The corresponding equitable partition will then be modified by taking
  also non-maximal sets into account.

  Suppose that all sets~$S_1,\dots, S_k$ are pairwise disjoint.
  From Lemma~\ref{lem:recursiveEBC} we derive an equicoloring~$P^+ \cup
  P^-$ of~$P$ with~$\delta_w = \card{P^+_w} -
  \card{P^-_w} \in \{0, \pm 1\}$ such that~$\sum_{w \in S_i} \delta_w \in
  \{0, \pm 1\}$ for all~$i \in [k]$.
  Note that the lemma only applies to the nodes~$w$ in~$S_i$ for which~$P_w
  \neq \emptyset$, however, it trivially extends to the general case.
  In the following, we extend this equicoloring of~$P$ to an
  equicoloring of~$(P,D)$ by assigning a suitable partition of~$D$.
  That is, we need to partition~$D$ such that
  \[
    \Delta_w
    =
    \begin{cases}
      \delta_w
      +
      \sum_{v \in V}(\card{D^-_{w,v}} - \card{D^+_{w,v}}),
      & \text{if } w \in U,\\
      \delta_w
      +
      \sum_{u \in V}(\card{D^+_{u,w}} - \card{D^-_{u,w}}),
      & \text{otherwise},
    \end{cases}
  \]
  is contained in~$\{0, \pm 1\}$ for every~$w \in V$.
  In particular, if~$w$ is a follower, there is a unique leader~$u$ such
  that its~$\Delta$-value reduces to~$\delta_w = \delta_w +
  \card{D^+_{u,w}} - \card{D^-_{u,w}}$ by the assumption that the
  sets~$S_i$ are pairwise disjoint.

  Note that, if~$\delta_w = 1$ for some follower~$w$, then we necessarily need
  to assign its leader-follower pair~$(u,w)$ to~$D^-$ to ensure~$\Delta_w = 1 + \card{D^+_{u,w}} -
  \card{D^-_{u,w}} \in \{0, \pm 1\}$.
  Analogously, if~$\delta_w = -1$ for some follower~$w$, then~$(u,w) \in D^+$.
  For followers~$w$ with~$\delta_w = 0$, however, we have two choices and
  we will specify later on how to assign these ordering inequalities
  to~$D^+$ and~$D^-$.
  Denote these not yet assigned inequalities (identified by their
  followers~$w \in V$) by~$\bar{D}$.

  Until now we have guaranteed that~$\Delta_w \in \{0, \pm 1\}$ for all~$w
  \in V \setminus (U \cup \bar{D})$.
  For a leader-follower pair~$(u,v)$, observe that assigning~$(u,v) \in D$
  with~$\delta_v = 1$ to~$D^-$  increases~$\Delta_u$ by~1, since~$u$ has a
  positive coefficient in the negated ordering inequality~$-(-x_u + x_v \leq 0)$;
  analogously, assigning~$(u,v) \in D$ with~$\delta_v = -1$ to~$D^+$
  decreases~$\Delta_u$ by~1.
  That is, for each~$u_i \in U$, the current assignment of~$D^+$ and~$D^-$
  implies
  \[
    \Delta_{u_i}
    =
    \delta_{u_i} + \sum_{\substack{v \in V\colon\\ (u_i,v) \in D}} \delta_v
    =
    \sum_{w \in S_i} \delta_w,
  \]
  where the last equation holds since~$\delta_w = 0$ for all~$w \in S_i$
  for which~$P_w = \emptyset$.
  By Lemma~\ref{lem:recursiveEBC}, we thus conclude that~$\Delta_{u_i} \in \{0,
  \pm 1\}$.

  To conclude the first case, we need to assign the ordering inequalities in~$\bar{D}$
  to~$D^+$ and~$D^-$.
  Since~$\delta_w = 0$ for each~$w \in \bar{D}$, we can assign them
  arbitrarily to~$D^+$ and~$D^-$ to achieve~$\Delta_w = \pm 1$.
  The only restriction we need to take into account is the coupling of all
  ordering inequalities via their corresponding leaders~$\Delta_u$.
  Because~$\Delta_u \in \{0, \pm 1\}$ if we do not consider~$\bar{D}$,
  we can easily maintain~$\Delta_u \in \{0, \pm 1\}$ by assigning
  the relevant ordering inequalities in~$\bar{D}$ alternatingly to~$D^+$ and~$D^-$.
  Consequently, $(C,D)$ admits an equicoloring and the assertion
  follows for the first case.

  Note that we can choose the alternating sequence such that is has the
  following property, which will be exploited in the remainder of the
  proof:
  For each~$S_i$, let~$R_1, \dots, R_j \subseteq S_i \setminus \{u_i\}$ be
  pairwise disjoint.
  Then, the alternating sequence of~$\bar{D}$ can be chosen such
  that~$\sum_{w \in R_\ell} \Delta_w \in \{0, \pm 1\}$ for all~$\ell \in
  [j]$.
  Indeed, this property holds, by first iterating over the elements
  in~$R_1 \cap \bar{D}$, then over the elements in~$R_2 \cap \bar{D}$, and
  so on.

  In the second case, suppose we have relabeled the sets~$S_i$ such
  that~$S_{\ell + 1},\dots, S_k$ are inclusionwise maximal.
  If we apply the previous arguments to~$S_{\ell + 1}, \dots, S_k$, we
  derive a bicoloring such that, for each~$i \in \{\ell + 1,
  \dots, k\}$ and~$w \in S_i \setminus \{u_i\}$,
  \[
    \Delta_w =
    \begin{cases}
      0, & \text{if } P_w \neq \emptyset,\\
      \pm 1, & \text{if } P_w = \emptyset.
    \end{cases}
  \]
  Moreover, if~$S_1, \dots, S_j$ are the inclusionwise maximal sets
  among~$S_1, \dots, S_\ell$, then no leader~$u_{\ell + 1}, \dots, u_k$ is
  contained in any of the sets~$S_1, \dots, S_j$ by~\ref{LRP1}.
  Thus, we can select the equicoloring such
  that~$\sum_{w \in S_i} \Delta_w \in \{0, \pm 1\}$ for all~$i \in [\ell]$
  by the previously derived property.

  We continue by assigning the ordering inequalities with leaders~$u_1,
  \dots, u_j$ to~$D^+$ and~$D^-$.
  Note that this does not change the~$\Delta$-value of any node~$w$
  outside~$\bigcup_{i = 1}^j S_i$.
  Again, if a follower~$w$ has~$\Delta_w = 1$, we need to assign its
  ordering inequality to~$D^-$ and if it has~$\Delta_w = -1$ to~$D^+$.
  As above, this gives the corresponding leader~$u_i$ a~$\Delta$-value
  of~$\sum_{w \in S_i} \Delta_w$, which is~$0$ or~$\pm 1$ by the derived
  property.
  Hence, we can assign the remaining ordering inequalities whose follower
  has~$\Delta_w = 0$ in an alternating order to~$D^+$ and~$D^-$ to
  maintain~$\Delta_{u_i} \in \{0, \pm 1\}$.
  We thus find an equitable partition such that, for all~$i \in [j]$,
  \[
    \Delta_w =
    \begin{cases}
      \pm 1, & \text{if } P_w \neq \emptyset,\\
      0, & \text{if } P_w = \emptyset.
    \end{cases}
  \]
  Using the same arguments as above, we can proceed iteratively until we
  also assigned the ordering inequalities of inclusionwise minimal sets
  in~$S_1,\dots,S_k$.
\end{proof}
Theorem~\ref{thm:stringentSST} is now a special case of
Theorem~\ref{thm:LRP} as we sketch next.

\begin{proof}[Proof of Theorem~\ref{thm:stringentSST}]
  We briefly sketch the proof's idea.
  If there is no edge contained in an orbit, the orbits form stable sets
  in~$G$.
  Moreover, the stabilizer computations guarantee that the inclusionwise
  maximal orbits are disjoint.
  Stringency implies that the inclusionwise maximal orbits have the
  recursion property.
  Since the SST leaders are not contained in succeeding orbits, the set of
  all orbits have the laminar recursion property.
  The result follows then by Theorem~\ref{thm:LRP}.
\end{proof}

\section{Preliminary Computational Results}
\label{sec:experiments}

In this section, we discuss the impact of SST presolving, cuts, and clique
cuts for the edge formulation of the maximum cardinality stable set problem.
Our test set consists of all graphs from the Color02 symposium~\cite{Color02} and
all complemented graphs from the max-clique DIMACS challenge~\cite{DIMACS} for
which we could find symmetries using \texttt{SageMath~9.1}~\cite{sagemath} within one
hour.
This gives us a test set of~105 graphs.
For all graphs, we computed at most~50 rounds of SST cuts, where we
selected an orbit of either minimal or maximal size; the leader is the
variable of smallest index in each orbit.

The left part of Table~\ref{tab:runningtimes} shows the proportion of nodes
and edges that remain in the graph after applying the deletion operation of
SST presolving.
Column ``edges+'' gives the proportion of edges after additionally applying
the addition operation.
SST cuts based on minimum orbits reduce the number of nodes
and edges by roughly~\SI{10}{\percent} and~\SI{20}{\percent}, respectively.
Selecting maximum orbits even reduces these quantities by~\SI{20}{\percent}
and~\SI{40}{\percent}; the biggest reduction can be achieved for the
instance \texttt{latin\_square\_10} from Color02, where the number of nodes drops
by~\SI{75}{\percent} and of edges even by~\SI{94}{\percent}.
Using the addition operation increases the number of edges by five
percentage points again.

\begin{table}[t]
  \begin{scriptsize}
    \caption{Comparison of effect of presolving of different SST variants.}
    \label{tab:runningtimes}
    \begin{tabular*}{\textwidth}{@{}l@{\;\;\extracolsep{\fill}}rrrrrrrrr@{}}\toprule
      & \multicolumn{3}{c}{graph reductions} & \multicolumn{6}{c}{solving times}\\
      \cmidrule{2-4} \cmidrule{5-10}
      orbit rule & nodes & edges & edges+ & presol & cut & clique & presol+ & cut+ & clique+\\
      \midrule
      minimum  & 0.90 & 0.81 & 0.85 & 0.55 & 0.56 & 0.56 & 0.43 & 0.50 & 0.47\\
      maximum  & 0.80 & 0.60 & 0.65 & 0.36 & 0.29 & 0.30 & 0.25 & 0.31 & 0.30\\
      \bottomrule
    \end{tabular*}
  \end{scriptsize}
\end{table}

In a second experiment, we investigated the impact of SST
presolving and cuts on running time.
These experiments have been conducted using \text{SCIP~8.0.0.1} (githash
a4eeac7) with \texttt{SoPlex~5.0.1.3} as LP solver; all symmetry handling
methods in \texttt{SCIP} have been disabled to get a fair comparison.
No time limit has been imposed and all experiments were run on a Linux
cluster with Intel Xeon E5 \SI{3.5}{\GHz} quad core processors and
\SI{32}{\giga\byte} memory.
It turns out that \texttt{SCIP} can solve most of these selected instances,
easily even without symmetry handling.
Therefore, we extracted the instances that need at least one second to be
solved, which leads to a reduced test set of~26 instances.

Without any symmetry handling, the geometric mean running time
is~\SI{10.3}{\second}.
The right part of Table~\ref{tab:runningtimes} shows the proportion of
solving time needed by the remaining methods for graphs obtained by the
deletion operation (presol), and additionally adding SST cuts (cut) or SST
clique cuts (clique).
The postfix ``+'' indicates that we additionally apply the addition
operation.
To generate clique cuts, we take the set of followers~$F$ of a leader
and greedily compute a clique covering of~$F$ within the subgraph induced by~$F$.
Again, the maximum orbit rule performs better.
Even just applying the deletion operation reduced the running time
by~\SI{64}{\percent}, adding either type of cuts reduces running time
by~\SI{70}{\percent}.
Additionally using the addition operation performs best and leads to a
running time reduction of~\SI{75}{\percent}.

%
\bibliographystyle{plain}
%

\end{document}